\documentclass{amsart}
\newtheorem{theorem}{Theorem}[section]
\newtheorem{lemma}[theorem]{Lemma}
\newtheorem{corollary}[theorem]{Corollary}

\newtheorem{example}[theorem]{Example}
\newtheorem{remark}[theorem]{Remark}
\def\erre{{\rm I\!R}}
\def\R{{\rm I\!R}}

\def\meas{\mathop{\rm meas}}

\def\R{{\rm I\!R}}
\def\R{{\rm I\!R}}

\def\di{\displaystyle}

\def\n{\noindent}
\def\phi{\varphi}
\def\div{\mathop{\rm div}}
\def\meas{\mathop{\rm meas}}

\newcommand{\dist}{\texttt{\rmfamily{dist}}}
\title[Multiple solutions for elliptic...]{Multiple solutions for elliptic equations involving a general operator in divergence form}
\author{Giovanni Molica Bisci}
\address[G. Molica Bisci]{University of Reggio Calabria, Department PAU\\
Via Melissari, 24, Reggio Calabria, Italy 89124} \email{gmolica@unirc.it}
\author{Du\v{s}an Repov\v{s}}
\address[D. Repov\v{s}]{Faculty of Education, and Faculty of Mathematics and Physics, University of Ljubljana, Kardeljeva pl. 16, Ljubljana, Slovenia 1000
}
\email{dusan.repovs@guest.arnes.si}
\thanks{Typeset by \LaTeX}
\begin{document}

\begin{abstract}
In this paper, exploiting variational methods, the existence of three weak solutions for a class of elliptic equations involving a general operator in divergence form and with Dirichlet boundary condition is investigated. Several special cases are analyzed. In conclusion, for completeness, a concrete example of an application is presented by finding the existence of three nontrivial weak solutions for an uniformly elliptic second-order problem on a bounded Euclidean domain.
\end{abstract}
\maketitle
\section{Introduction}
In this paper we study the existence of multiple weak solutions for the following Dirichlet boundary-value problem:
\begin{equation} \tag{$D_\lambda^{f}$} \label{N}
\left\{
\begin{array}{ll}
-\operatorname{div}(a(x,\nabla u))= \lambda k(x)f(u) & \rm in \quad \Omega
\\ u=0  & \rm on \,
\partial \Omega,\\
\end{array}
\right.
\end{equation}
where $\Omega$ is a bounded domain in ${\R}^{N}$ ($N\geq 2$) with smooth
boundary $\partial \Omega$, $p>1$, $a:\bar\Omega\times {\R}^{N}\to {\R}^N$ is a suitable continuous map of gradient type,
 and $\lambda$ is a positive real parameter. Further, $f:\erre\rightarrow \erre$ and $k:\bar\Omega\rightarrow \R^{+}$ are two continuous functions.\par

\indent We first observe that our multiplicity theorems are related to some known results in the current literature on the subject.
For instance, we refer the reader to the paper of Br\'{e}zis and Oswald \cite{BO} in which an existence and uniqueness result was obtained via the minimization technique
and by using a maximum principle, as well as to the work \cite{lin}, in which Lin proved the existence,
uniqueness and asymptotical properties of the solutions of problem \eqref{N}, when $f$ behaves
like the power $t^q$, for some $q\in ]0,1[$, and sufficiently large $t$.\par

 Further, Dirichlet
problems involving a general operator in divergence form was studied by De N\'apoli and Mariani in \cite{n1}. In this paper the existence of one weak solution were proved by exploiting the standard
mountain pass geometry and requiring, among other assumptions, that the nonlinearity $f$ has a $(p-1)$-superlinear behaviour at infinity. The non-uniform case was successively considered by Duc and
Vu in \cite{d1}, by extending the result of \cite{n1} under the key hypothesis that the map $a$ fulfills the growth condition:
$$
|a(x,\xi)|\leq c_0(h_0(x)+h_1(x)|\xi|^{p-1}),
$$
for every $(x,\xi)\in \Omega\times\erre^N$, where $h_0\in L^{p/(p-1)}(\Omega)$, $h_1\in L^{1}_{loc}(\Omega)$ and $c_0$ is a positive constant.\par
In both papers \cite{n1} and \cite{d1} the function $f$ satisfies the celebrated Ambrosetti-Rabinowitz condition:
\begin{itemize}
\item[(AR)]
\textit{There exist $s_0>0$ and $\theta>p$ such that
$$
0<\theta \int_0^sf(t)\;dt\leq s f(s),\,\,\,\,\,\,\forall\, |s|\geq s_0.
$$}
\end{itemize}

In \cite{CPV}, Colasuonno, Pucci and Varga studied different and very general classes of elliptic operators in divergence form
looking at the existence of multiple weak solutions; see also Remark \ref{onAR1}. Their contributions represent a nice improvement, in several
directions, of the results obtained
 by Krist\'aly,
Lisei and Varga in \cite{k1} in which a uniform Dirichlet problem with parameter is investigated.\par
 Yang, Geng and Yan, in \cite{y1}, proved the existence of three weak solutions for singular
$p$-Laplacian type equations. Further, Papageorgiou, Rocha and Staicu considered a nonsmooth $p$-Laplacian problem in divergence form, obtaining
the existence of at least two nontrivial weak solutions (see \cite{p1}).\par

We also note that Autuori, Pucci and Varga, in \cite{APV}, studied a quasilinear elliptic eigenvalue problem, for every value of the parameter $\lambda$, by using suitable variational techniques (see \cite{PePuVa} for related topics).\par
 A similar approach is adopted by Servadei in \cite{Raffy}.
 More precisely, in this paper, the author consider the following variational inequality involving a self-adjoint, second-order, uniformly elliptic operator $A$ without lower-order terms:
\vskip2pt \n {\it Find $u\in K$ such that
\begin{eqnarray*}\label{norm}
<Av,v-u> &-& \lambda \int_\Omega u(x)(v-u)(x)dx\\
               &\geq & \int_\Omega p(x)f(u(x))(v-u)(x)dx\nonumber.
\end{eqnarray*}
 for all $v\in K$},
\vskip2pt \n
where $\lambda$ is a positive parameter, $p\in L^\infty(\Omega)$ vanishes only on a set of measure zero and positive somewhere on $\Omega$, $f$ is a continuous function satisfying certain superlinear and subcritical growth conditions at zero and at infinity, and $K$ is a suitable subset of $H^1_0(\Omega)$.\par
 If $A$ is the Laplacian and if $f$ has a special form, then this problem is well understood, and the behavior of the solution depends on the value of $\lambda$ compared to the first eigenvalue of the Laplacian. In the cited paper, the author extends these results to a general setting. In particular, if $\lambda_1$ denotes the first eigenvalue of $A$, then there is a nontrivial nonnegative solution if $\lambda<\lambda_1$. The method of proof is based on the Mountain Pass Lemma. Further, the Linking Theorem was applied in the case $\lambda\geq \lambda_1$ (see \cite{Molica} for related applications).\par

Motivated by this large interest in elliptic equations with operators in divergence form, the aim of this paper is to establish, by using
variational arguments (see Theorem \ref{BM}), a precise interval of values of the parameter
$\lambda$ for which problem $(D_{\lambda}^{f})$ admits at least three
weak solutions, without explicit perturbations of the
nonlinear term, even in the higher dimensional setting (see Theorem \ref{Main1}).\par

A direct consequence of Theorem \ref{Main1} is presented in Corollary \ref{main3}. More precisely, in this case the existence of at least two weak nontrivial solutions for problem $(D_{\lambda}^{f})$ is obtained for sufficiently large $\lambda$. A special case of Corollary \ref{main3} is Corollary \ref{main8}, where the classical autonomous setting is stated (see Example \ref{inintroduction}).

The plan of the paper is as follows. Section 2 is devoted to our abstract framework, while Section 3 is dedicated to the main results. A concrete example of an application is then presented (see Example \ref{esempio}).
 Finally, we cite a recent monograph by Krist\'aly, R\u adulescu and Varga \cite{k2} as a general reference on variational methods adopted here.

\section{Abstract Framework}

Let $\Omega$ be a bounded domain in ${\erre}^N$  with smooth
boundary $\partial \Omega$, $p>1$ and denote by:

\begin{itemize}
\item[$\circ$] $X$ the Sobolev space $W_{0}^{1,p}(\Omega)$ endowed by the norm
$$
\|u\|:= \Big(\int_{\Omega} |\nabla u(x)|^p dx\Big)^{1/p};
$$
\item[$\circ$] $X^*$ the topological dual of $X$;
\item[$\circ$] $\langle\cdot,\cdot\rangle$ the duality brackets for the pair $(X^{\ast},X)$.
\end{itemize}
In addition, as customary, the symbol
$$
p^*:=\left\{
\begin{array}{ll}
\displaystyle{pN}/{(N-p)} & {\rm if} \quad 1<p<N
\\ \infty  & {\rm if} \quad
p\geq N,\\
\end{array}
\right.
$$
denotes the critical Sobolev exponent of $p$.
Fixing
$q\in[1,p^*]$, by the Sobolev embedding theorem, there
exists a positive constant $c_q$ such that
\begin{equation}\label{Poincare2}
\|u\|_{L^q(\Omega)}\leq c_q \|u\|\, ,\quad\, u\in X,
\end{equation}
\noindent and, in particular, the embedding $X\hookrightarrow L^q(\Omega)$ is compact for every $q\in [1,p^*[$. Moreover, if $1<p<N$, the best constant $c_{p^*}$ is given by
\begin{equation}\label{TA}
c_{p^*}=\frac{1}{N\sqrt{\pi}}\left(\frac{N!\Gamma\left(\di\frac{N}{2}\right)}{2\Gamma\left(\di\frac{N}{p}\right)\Gamma\left(\di N+1-\frac{N}{p}\right)}\right)^{1/N}\eta^{1-1/p},
\end{equation}
where
$$
\eta:=\displaystyle\frac{N(p-1)}{N-p},
$$
\noindent see, for instance, \cite{TA}. If $p>N$, let
\[
\sup\left\{\frac{\max_{x\in\overline\Omega}|u(x)|}{\|u\|}: u\in
W^{1,p}_0(\Omega), u\neq 0 \right\}<+\infty.
\]
\noindent It is well-known
(\cite[formula (6b)]{TA}) that by putting
\begin{equation} \label{dT}
m:= \frac{N^{-\frac{1}{p}}}{\sqrt{\pi}}\left[\Gamma\left(1+\frac{N}{2}\right)\right]^{\frac{1}{N}}\left(\frac{p-1}{p-N}\right)^{1-\frac{1}{p}}(\meas(\Omega))^{\frac{1}{N}-\frac{1}{p}},
\end{equation}
one has
\begin{equation} \label{dT2}
\|u\|_\infty:=\max_{x\in\bar\Omega}|u(x)|\leq m\|u\|,
\end{equation}
\noindent for every $u\in X$ (equality
occurs when $\Omega$ is a ball). Here $\Gamma$ is the Gamma function defined by
$$
\Gamma(t):=\int_0^{+\infty}z^{t-1}e^{-z}dz,\,\,\,\, \forall\; t>0,
$$
and
$``\meas(\Omega)"$ denotes the usual Lebesgue measure of $\Omega$. Moreover, let
 $$
 \tau:=\sup_{x\in\Omega}\dist(x,\partial\Omega).
 $$
 \indent Simple calculations show that there is $x_0 \in \Omega$ such that
 $B(x_0,\tau) \subseteq \Omega$, where $B(x_0,\tau)$ is the open ball of radius $\tau$ centered at the point $x_0$. We also denote by
$$
\omega_s:=s^N\frac{\pi^{N/2}}{\Gamma \Big(\displaystyle 1+\frac{N}{2}\Big)},
$$
the measure of the $N$-dimensional ball of radius $s>0$.
\noindent At this point, for $\delta>0$, let $u_\delta\in X$ be the following function
\[
u_\delta(x):= \left\{
\begin{array}{ll}
0 & \mbox{ if $x \in \Omega \setminus B(x_0,\tau)$} \\\\
\displaystyle\frac{2\delta}{\tau} \left(\tau- |x-x_0| \right)
& \mbox{ if $x \in B(x_0,\tau) \setminus B(x_0,{\tau}/{2})$} \\\\
\delta & \mbox{ if $x \in B(x_0,{\tau}/{2}),$}
\end{array}
\right.
\]
\noindent that will be useful in the sequel in the proof of our theorems. One has that
\begin{equation*}
\begin{aligned}
\|u_\delta\|^p=\int_\Omega|\nabla u_\delta(x)|^p \;dx= \frac{2^p\delta^p\omega_\tau}{\tau^p}\left(1-\frac{1}{2^N}\right).
\end{aligned}
\end{equation*}
\noindent Indeed,
\begin{equation*} \label{embed}
\begin{aligned}
\int_\Omega |\nabla u_\delta(x)|^p \;dx
&=\delta^p\int_{B(x_0,\tau)
\setminus B(x_0,\tau/2)}\frac{2^p}{\tau^p}\;dx\\
& =\frac{2^p\delta^p}{\tau^p}(\meas(B(x_0,\tau))-\meas(B(x_0,{\tau/2})))\\
&=\frac{2^p\delta^p\omega_\tau}{\tau^p}\left(1-\frac{1}{2^N}\right),
\end{aligned}
\end{equation*}
where, from now on, $``\meas(B(x_0,s))"$ for $s>0$ stands for the Lebesgue measure of the open ball $B(x_0,s)$.\par
 As pointed out before, our approach is variational. More precisely, the main tool will be the following abstract critical point theorem for smooth functions that can be derived from \cite[Theorem 3.6]{BONMAR}.

\begin{theorem}\label{BM}{Let $X$ be a reflexive real Banach space and let $\Phi, \Psi :X\to\R$ be two $C^1$-functionals with $\Phi(0_X)=\Psi(0_X)=0$ and such that$:$
\begin{itemize}
	\item[$({\rm a}_1)$] $\Phi$ is coercive and sequentially weakly lower semicontinuous$;$
	\item[$({\rm a}_2)$] $\Psi :X\to \R$ is sequentially weakly upper semicontinuous.
\end{itemize}
For every $\lambda>0$, put $J_{\lambda}:=\Phi - \lambda \Psi$ and assume that the following conditions are satisfied$:$
\begin{itemize}
\item[$({\rm a}_3)$] There exist $r>0$ and $\bar{u}\in X$, with $r<\Phi(\bar{u})$, such that$:$
$$
\displaystyle\frac{\displaystyle\sup_{u\in\Phi^{-1}(]-\infty,r])}\Psi(u)}{r}< \frac{\Psi(\bar{u})}{\Phi(\bar u)};
$$
\item[$({\rm a}_4)$] For each $\lambda\in \Lambda_{r}:=\left]\displaystyle\frac{\Phi(\bar u)}{\Psi(\bar{u})},\displaystyle\frac{r}{\displaystyle\sup_{u\in\Phi^{-1}(]-\infty,r])}\Psi(u)}\right[$ the functional $J_{\lambda}$ is bounded from below and fulfills $($\rm{PS}$)_{\mu}$, with $\mu\in\R$.
\end{itemize}
Then for each $\lambda \in \Lambda_{r}$, the functional $J_{\lambda}$ has at least three distinct critical points in $X$.
}
\end{theorem}

For the sake of completeness, we also recall that the $C^1$-functional $J_\lambda:X\to\R$ satisfies the Palais-Smale condition at level $\mu\in\R$ when
\begin{itemize}
\item[$\textrm{(PS)}_{\mu}$] {\it Every sequence $\{u_n\}\subset X$  such that
$$
J_\lambda(u_n)\to \mu, \qquad{\it and}\qquad \|J'_\lambda(u_n)\|_{X^*}\to0,
$$
as $n\rightarrow \infty$, possesses a convergent subsequence in $X$.}
\end{itemize}

\section{Main results}

 In the sequel, let $p>1$ and let $\Omega\subset \erre^N$ be a bounded Euclidean domain, where $N\geq 2$. Further, let $A:\bar\Omega\times \erre^{N}\to \erre$, and let $A=A(x,\xi)$ be a continuous function in $\bar\Omega\times \erre^{N}$, with continuous gradient $a(x,\xi):=\nabla_\xi A(x,\xi):\bar \Omega\times \erre^N\rightarrow \erre^{N}$, and assume
that the following conditions hold:
\begin{itemize}
\item[$(\alpha_1)$] $A(x,0)=0$ \textit{for all} $x\in\Omega$;

\item[$(\alpha_2)$] $a$ \textit{satisfies the growth condition}
$|a(x,\xi)|\leq c(1+|\xi|^{p-1})$ \textit{for all} $x\in\Omega$, $\xi\in
\erre^{N}$, \textit{for some constant} $c>0$;

\item[$(\alpha_3)$] $A$ is $p$-\textit{uniformly convex, that is}
\[
A\left(x,\frac{\xi+\eta}{2}\right)\leq\frac{1}{2}A(x,\xi)
+\frac{1}{2}A(x,\eta)-k|\xi-\eta|^{p},
\]
\textit{for every }$x\in\bar\Omega,$ $
\xi,\eta\in \erre^{N}$ \textit{and some} $k>0$;

\item[$(\alpha_4)$] $A$ \textit{is} $p$-\textit{subhomogeneous i.e.},
$0\leq a(x,\xi)\cdot\xi\leq p A(x,\xi)$  for all $x\in\bar\Omega$, $\xi\in
\erre^{N}$;

\item[$(\alpha_5)$] $A$ \textit{satisfies}
$\Lambda_1 |\xi|^{p} \leq A(x,\xi)\leq \Lambda_2 |\xi|^{p}$  \textit{for all} $x\in\bar\Omega$, $\xi\in
\erre^{N}$, \textit{where} $\Lambda_1$ \textit{and} $\Lambda_2$ \textit{are positive constants}.
\end{itemize}

\begin{example}\rm{
 Note that by choosing $$
 A(x,\xi):=\displaystyle\frac{|\xi|^p}{p},$$ with $p\geq 2$, we have the usual $p$-Laplacian operator. See also Remark \ref{diriclet} for details on this special case.
}
\end{example}

\begin{remark}\label{commento}\rm
Define
$A^{|\vee|}:\bar\Omega\times\erre\to\erre$ as follows,
\[
A^{|\vee|}(x,t):=\sup_{|\xi|=t}A(x,\xi), \quad \forall\; x\in\bar\Omega.
\]
For every $\varepsilon, b\in (0,1)$ and $x\in\bar\Omega$, let us put $E_{\varepsilon,b}(x)$ to be the set
\begin{align*}
\Bigg\{&(\xi,\eta)\in\erre^{N}\times\erre^{N}:
A\left(x,\frac{\xi-\eta}{2}\right)\geq\frac{1}{2}\max\left\{A(x,\varepsilon\xi),
A(x,\varepsilon\eta)\right\}, \\
&A\left(x,\frac{\xi+\eta}{2}\right)>(1-b)\frac{A(x,\xi)+A(x,\eta)}{2}\Bigg\},
\end{align*}
and
\[
q_{\varepsilon,b}(x):=\sup\left\{\frac{|\xi-\eta|}{2}: (\xi,\eta)\in
E_{\varepsilon,b}(x)\right\}.
\]
It is known that the map $A$ is said to be \textit{uniformly convex} if it satisfies
\[
\lim_{b\to0}\int_{\Omega}A^{|\vee|}(x,q_{\varepsilon,b}(x))dx=0,
\quad \text{for every } \varepsilon\in(0,1).
\]
\par

\noindent As pointed out in \cite[Remark 2.2]{f1}, we just observe that our hypotheses $(\alpha_1)$-$(\alpha_5)$ imply that $A$ is an uniformly convex operator.
\end{remark}

\begin{remark}\label{S+}\rm

Integrating $(\alpha_2)$ we deduce that the map $A$ satisfies
the growth condition:
$$
A(x,\xi)\leq c(|\xi|+|\xi|^{p}),
$$
 for every $x\in\bar\Omega$ and $\xi\in
\erre^{N}$. Indeed, we have
$$
A(x,\xi)=\int_0^1\frac{d}{dt}A(x,t\xi)dt=\int_0^1a(x,t\xi)\cdot \xi dt.
$$
Hence, we deduce that
$$
A(x,\xi)\leq c\int_0^1(1+|\xi|^{p-1}t^{p-1})|\xi|dt\leq c|\xi|+\frac{c}{p}|\xi|^p\leq c(|\xi|+|\xi|^{p}),
$$
for every $x\in\bar\Omega$ and $\xi\in
\erre^{N}$.
Moreover, we also note that by condition $(\alpha_3)$, the functional $\Phi: X\rightarrow \erre$ defined by
 $$
  u\mapsto \int_{\Omega}A(x,\nabla u(x))dx,
 $$
 is a locally uniformly convex operator. Taking into account the above facts, its G\^{a}teaux derivative $\Phi': X\rightarrow X^*$ satisfies the $(S_+)$ condition; see \cite[Proposition 2.1]{n1}:

 \begin{itemize}
\item[$(S_+)$] {\it For every sequence $\{u_n\}\subset X$ such that $u_n\rightharpoonup u$ $($weakly\,$)$ in $X$ and
$$
\limsup_{n\to\infty}\int_{\Omega}a(x,\nabla u_n(x))
 \cdot\nabla(u_n-u)(x)dx\leq0,$$
 it follows that $u_n\rightarrow u$ $($strongly\,$)$ in $X$.}
\end{itemize}

\end{remark}

\indent From now on, we say that a continuous function $f:\erre\rightarrow\erre$ belongs to the class $\mathfrak{F}_q$ if
\begin{equation}\label{S}
|f(t)|\leq a_1+a_2|t|^{q-1},\,\,\,\,\,(\forall\; t\in\erre)
\end{equation}
for some nonnegative constants $a_1,a_2$, where $q\in ]1,
pN/(N-p)[$ if $p<N$ and $1<q<+\infty$ if $p\geq N$.\par
\indent Let $\Psi:X\rightarrow \erre$ be defined
$$
\Psi(u):=\int_{\Omega}k(x)F(u(x))dx,$$
where $k:\bar\Omega\rightarrow \erre$ is a positive and continuous function, and
$$
F(s):=\int_0^{s}f(t)dt,
$$
for every $s\in\erre$.\par

Fixing $\lambda>0$, we say that $u\in X$ is a \textit{weak solution} of problem \eqref{N} if
$$
\int_{\Omega}a(x,\nabla u(x))\cdot\nabla v(x)dx=\lambda \int_\Omega k(x)f(u(x)) v(x) dx,
$$

\noindent for every $v\in X$.\par
 In the sequel we look at the existence of weak solutions of problem \eqref{N} finding critical points for a suitable associated energy functional $J_\lambda$.

\begin{lemma}\label{lemma}
 Let us assume that $f\in\mathfrak{F}_q$. Then the functionals $\Phi$ and $\Psi$ are respectively sequentially weakly lower and upper semicontinuous. Hence in particular,
 for every $\lambda\in\erre^+$, the functional
$J_\lambda:X\to \erre$ given by
$$
J_\lambda:=\Phi-\lambda\Psi,
$$
is sequentially weakly lower semicontinuous.
\end{lemma}
\begin{proof}
 The functional $\Phi$ being locally uniformly convex,
is weakly lower semicontionous. On the other hand, since $f\in\mathfrak{F}_q$ and $k\in C^{0}(\bar \Omega)$, we have
$|k(x)f(t)|\leq \|k\|_\infty(a_1+a_2|t|^{q-1})$, for every $(x,t)\in\bar\Omega\times\erre$. Finally, due to the fact that
the embedding $X \hookrightarrow
L^{q}(\Omega)$ is compact, we obtain that $\Psi$ is
sequentially weakly (upper) semicontinuous in the standard way.
\end{proof}
\indent Set
$$
\kappa:=\left(\frac{2^{N-p}}{\omega_\tau(2^N-1)\Lambda_1}\right)^{1/p}\tau,
$$
\noindent and
$$
G_1:=\left(\frac{2^p(2^N-1)\Lambda_2}{\tau^p\displaystyle\min_{x\in\bar\Omega}k(x)}\right)\frac{\|k\|_{\infty}c_1}{\Lambda_1^{1/p}},
$$
as well as
$$
G_2:=\left(\frac{2^p(2^N-1)\Lambda_2}{\tau^p\displaystyle\min_{x\in\bar\Omega}k(x)}\right)\frac{\|k\|_{\infty}c_q^q}{q\Lambda_1^{q/p}}.
$$
\indent Our main result is as follows.
\begin{theorem} \label{Main1}
Let $f\in \mathfrak{F}_q$ and assume that
\begin{itemize}
\item [$(\textrm{h}_0)$]$F(s)\geq 0$ for every $s\in\erre^+$$;$
\item [$(\textrm{h}_1)$]
$
\displaystyle\lim_{|t|\rightarrow \infty}\frac{f(t)}{|t|^{p-1}}=0;
$
\item [$(\textrm{h}_2)$] There are positive constants $\gamma$ and $\delta$, with $\delta>\gamma\kappa$ such that
$$
\frac{F(\delta)}{\delta^p}>a_1\frac{G_1}{\gamma^{p-1}}+a_2G_2\gamma^{q-p}.
$$

\end{itemize}

\noindent Then for each parameter $\lambda$ belonging to
$$
\Lambda_{(\gamma,\delta)}:=\frac{2^p(2^N-1)\Lambda_2}{\tau^p\displaystyle\min_{x\in\bar\Omega}k(x)}\left]\frac{\displaystyle \delta^p}{\displaystyle  F(\delta)},\frac{1}{\Big(a_1\displaystyle\frac{G_1}{\gamma^{p-1}}+a_2G_2\gamma^{q-p}\Big)}\right[,
$$
\noindent problem
\eqref{N} possesses at least three weak solutions.
\end{theorem}
\begin{proof}
Our aim is to apply Theorem \ref{BM}. Let $X:=W^{1,p}_0(\Omega)$ and consider the
functionals $\Phi,\Psi:X\to\R$ defined before. Clearly, $\Phi:X\to\R$ is coercive since, by condition $(\alpha_5)$, it follows that
$$
\Phi(u)\geq \Lambda_1\|u\|^{p}\rightarrow +\infty,
$$
when $\|u\|\rightarrow \infty$. Moreover, $\Phi$ is a continuously G\^{a}teaux differentiable and sequentially weakly lower semicontinuous functional (by Lemma \ref{lemma}). On the other
hand, $\Psi$ is
well-defined, continuously G\^{a}teaux differentiable and, again by Lemma \ref{lemma}, a weakly upper semicontinuous functional. More precisely, one has
$$
\Phi'(u)(v)=\int_{\Omega}a(x,\nabla u(x))\cdot\nabla
v(x)dx,$$
$$
\Psi'(u)(v)=\int_\Omega k(x)f(u(x)) v(x) dx,
$$
for every $u,v\in X$. Hence conditions $({\rm a}_1)$ and $({\rm a}_2)$ of Theorem \ref{BM} are satisfied. Now, fix $\lambda>0$. A critical point of the functional $J_\lambda:=\Phi-\lambda\Psi$ is a
function $u\in X$ such that
$$
\Phi'(u)(v)-\lambda\Psi'(u)(v)=0,
$$
for every $v\in X$. Hence, the critical points of the functional $J_\lambda$ are weak solutions of problem \eqref{N}. Moreover, $\Phi (0_X)=\Psi (0_X)=0.$
 \noindent Since $f\in \mathfrak{F}_q$, one has that
\begin{equation}\label{inequality}
F(s)\leq a_1|s|+a_2\frac{|s|^{q}}{q},
\end{equation}
for every $s\in \erre$.\par
\noindent Let $r\in ]0,+\infty[$ and consider the function
$$
\chi(r):=\frac{\displaystyle\sup_{u\in\Phi^{-1}(]-\infty,r])}\Psi(u)}{r}.
$$
\noindent Taking into account (\ref{inequality}), it follows that
$$
\Psi(u)=\int_\Omega k(x)F(u(x))dx\leq \|k\|_\infty\left(a_1\|u\|_{L^1(\Omega)}+\frac{a_2}{q}\|u\|_{L^q(\Omega)}^q\right).
$$
Then let $u\in X$ be such that $\Phi(u)\leq r$, that is
$$
\int_{\Omega}A(x,\nabla
u(x))dx\leq r.
$$
Hence the above relation together with condition $(\alpha_5)$ implies that
$$
\|u\|\leq \frac{r^{1/p}}{\Lambda^{1/p}_1}.
$$
In other words, one has the following algebraic inclusion
$$
\left\{u\in X:\int_{\Omega}A(x,\nabla
u(x))dx\leq r\right\}\subseteq \left\{u\in X:\|u\|\leq \frac{r^{1/p}}{\Lambda^{1/p}_1}\right\}.
$$
Consequently, by the Sobolev inequalities (\ref{Poincare2}), one has
$$
\Psi(u)\leq \|k\|_\infty\left(\frac{c_1a_1r^{1/p}}{\Lambda^{1/p}_1}+\frac{c_q^qa_2r^{q/p}}{q \Lambda^{q/p}_1}\right),
$$
for every $u\in \Phi^{-1}(]-\infty,r])$.
Hence,
$$
\sup_{u\in\Phi^{-1}(]-\infty,r])}\Psi(u)\leq \|k\|_\infty \left(\frac{c_1a_1r^{1/p}}{\Lambda^{1/p}_1}+\frac{c_q^qa_2r^{q/p}}{q \Lambda^{q/p}_1}\right).
$$
\noindent By the above inequality, the following relation holds
\begin{equation}\label{n}
\chi(r)\leq \|k\|_\infty L(r),
\end{equation}
\noindent for every $r>0$, where we set
$$
L(r):=\frac{c_1a_1r^{1/p-1}}{\Lambda^{1/p}_1}+\frac{c_q^qa_2r^{q/p-1}}{q\Lambda^{q/p}_1}.
$$
\noindent Next, let $u_\delta$ be the function defined in Section 2. Clearly $u_\delta\in X$ and, since $(\alpha_5)$ holds, we have
\begin{eqnarray}\label{norm}
\Phi(u_\delta) &=&\int_{\Omega}A(x,\nabla
u_\delta(x))dx\\
               &\leq & \Lambda_2\left(\frac{2^p\delta^p\omega_\tau}{\tau^p}\left(1-\frac{1}{2^N}\right)\right)\nonumber.
\end{eqnarray}
Taking into account that $\delta>\gamma\kappa$, by a direct computation,
one has 
$\gamma^p<\Phi(u_\delta)$. Moreover from $(\textrm{h}_0)$ and taking into account that the map $k$ is positive and continuous in $\bar\Omega$, we easily derive
\begin{eqnarray}\label{l}
\int_\Omega k(x)F(u_\delta(x))\;dx\geq F(\delta)\frac{\pi^{N/2}}{\Gamma(1+N/2)}\frac{\tau^N}{2^N}\min_{x\in\bar\Omega} k(x).
\end{eqnarray}
\noindent Hence, by (\ref{norm}) and (\ref{l}), one has
\begin{equation}\label{condition}
\frac{\Psi(u_\delta)}{\Phi(u_\delta)}\geq \left(\frac{\displaystyle \tau^p\min_{x\in\bar\Omega} k(x)}{2^p(2^{N}-1)\Lambda_2}\right)\frac{F(\delta)}{\delta^p}.
\end{equation}
In view of $(\textrm{h}_2)$ and taking into account (\ref{n}) and (\ref{condition}), we get
\begin{eqnarray}
\chi(\gamma^p)=\frac{\displaystyle\sup_{u\in\Phi^{-1}(]-\infty,\gamma^p])}\Psi(u)}{\gamma^p}
&\leq & \|k\|_\infty L(\gamma^p)\nonumber \\
& =&  \frac{\displaystyle \tau^p\min_{x\in\bar\Omega} k(x)}{2^p(2^{N}-1)\Lambda_2}\Big(a_1\displaystyle\frac{G_1}{\gamma^{p-1}}+a_2G_2\gamma^{q-p}\Big)\nonumber\\
               & <& \left(\frac{\displaystyle \tau^p\min_{x\in\bar\Omega} k(x)}{2^p(2^{N}-1)\Lambda_2}\right)\frac{F(\delta)}{\delta^p}\nonumber \\
               & \leq& \frac{\Psi(u_\delta)}{\Phi(u_\delta)}.
               \nonumber
\end{eqnarray}
\noindent Therefore, assumption $({\rm a}_3)$ of Theorem \ref{BM} is satisfied.\par
\noindent Let us fix $\lambda>0$. By condition $(\textrm{h}_1)$, there exists $\delta_\lambda$ such that
$$
|f(t)|\leq \frac{\Lambda_1}{c_p^p(1+\lambda)\|k\|_\infty}|t|^{p-1},
$$
\noindent for every $|t|\geq \delta_\lambda$. By integration we get
$$
|F(s)|\leq \frac{\Lambda_1}{c_p^p(1+\lambda)\|k\|_\infty}|s|^{p}+\max_{|t|\leq \delta_\lambda}|f(t)||s|,
$$
\noindent for every $s\in\erre$.\par
\noindent Thus
\begin{eqnarray}\label{e3.2a}
J_\lambda(u) &\geq& \Phi(u)-\lambda |\Psi(u)|\nonumber\\
               &\geq& \Lambda_1\|u\|^p-\frac{\Lambda_1\lambda}{(1+\lambda)}\|u\|^p-c_1\|k\|_{\infty}\lambda \max_{|t|\leq \delta_\lambda}|f(t)|\|u\|.\nonumber
\end{eqnarray}
\noindent Then the functional $J_\lambda$ is bounded from below and, since $p>1$,
$J_\lambda(u)\rightarrow +\infty$ whenever $\|u\|\rightarrow +\infty$. Hence $J_\lambda$ is coercive. Now, let us prove that $J_\lambda$ satisfies the condition $\textrm{(PS)}_{\mu}$ for $\mu\in\erre$. For our goal, let $\{u_{n}\}\subset X$ be a Palais-Smale
sequence, i.e.
$$
J_\lambda(u_n)\to \mu, \qquad{\rm and}\qquad \|J'_\lambda(u_n)\|_{X^*}\to0.
$$
Taking into account the coercivity of $J_\lambda$, the sequence $\{u_n\}$ is necessarily bounded in $X$. Since $X$ is reflexive, we may extract a subsequence,
which for simplicity we call again $\{u_{n}\}$, such that
$u_{n}\rightharpoonup u$ in $X$. We will prove that $\{u_{n}\}$ strongly converges to $u\in X$.
Exploiting the derivative $J_\lambda'(u_n)(u_n-u)$, we obtain
\begin{eqnarray*}
\int_{\Omega}a(x,\nabla
u_n(x))\cdot\nabla(u_n-u)(x)dx &=& \langle J'_\lambda(u_n),u_n-u\rangle\\
               &+ & \lambda\int_{\Omega}k(x)f(u_n(x))(u_n-u)(x)dx.\nonumber
\end{eqnarray*}
Since $\|J'_\lambda(u_n)\|_{X^*}\to0$ and the sequence $\{u_n-u\}$ is bounded in
$X$, taking into account that $|\langle
J'_\lambda(u_n),u_n-u\rangle|\leq\|J'_\lambda(u_n)\|_{X^*}\|u_n-u\|$, one has
\[
\langle J'_\lambda(u_n),u_n-u\rangle\to0.
\]
 Further, by the asymptotic condition $(\textrm{h}_1)$, there exists a real positive constant $c$ such that $|f(t)|\leq c(1+|t|^{p-1})$, for every $t\in\erre$. Then
\begin{align*}
&\int_{\Omega}k(x)|f(u_n(x))||u_n(x)-u(x)|dx \\
&\leq \sigma\int_{\Omega}|u_n(x)-u(x)|dx
+ \sigma\int_{\Omega}|u_n(x)|^{p-1}|u_n(x)-u(x)|dx \\
&\leq \sigma((\meas(\Omega))^{1/p'}+\|u_n\|_{L^{p}(\Omega)}^{p-1}) \|u_n-u\|_{L^p(\Omega)},
\end{align*}
where $\sigma:=c\|k\|_\infty$ and $p'$ is the conjugate exponent of $p$.\par
 Now, the embedding
$W^{1,p}(\Omega)\hookrightarrow L^p(\Omega)$ is compact, hence $u_n\to u$ strongly in $L^p(\Omega)$.
So we obtain
\[
\int_{\Omega}k(x)|f(u_n(x))||u_n(x)-u(x)|dx\to0.
\]
 We may conclude
\[ 
\limsup_{n\to\infty}\langle a(x,u_{n}),u_{n}-u\rangle
=0,
\]
where $\langle a(x,u_{n}),u_{n}-u\rangle$ denotes
$$\displaystyle\int_{\Omega}a(x,\nabla u_n(x))\cdot\nabla(u_n-u)(x)dx.$$
\noindent But as observed in Remark \ref{S+}, the operator $\Phi'$ has the $(S_+)$ property. So, in conclusion, $u_{n}\to u$ strongly in $X$. Hence, $J_\lambda$ is bounded from below and fulfills $($\rm{PS}$)_{\mu}$ (with $\mu\in\R$), for every positive parameter, in particular, for every $$\lambda\in \Lambda_{(\gamma,\delta)}\subseteq \left]\displaystyle\frac{\Phi(u_\delta)}{\Psi(u_\delta)},\displaystyle\frac{\gamma^p}{\displaystyle\sup_{u\in \Phi^{-1}(]-\infty,\gamma^p])}\Psi(u)}\right[.$$ Then also condition $({\rm a}_4)$ holds. Hence all the assumptions of Theorem \ref{BM} are satisfied. Consequently, for each $\lambda \in \Lambda_{(\gamma,\delta)}$, the functional $J_\lambda$ has at least three distinct critical points that are weak solutions of the problem \eqref{N}.
\end{proof}

\begin{remark}\label{diriclet0}\rm A careful analysis of the proof of Theorem \ref{Main1} shows that condition $(\textrm{h}_0)$ can be replaced by a more general sign condition on the potential:
\begin{itemize}
\item [$(\textrm{h}_0')$]$F(s)\geq 0$, \textit{for every} $s\in ]0,\delta[$$,$
\end{itemize}
where $\delta$ is the constant that appears in hypothesis $(\textrm{h}_2)$.
\end{remark}

\begin{remark}\label{diriclet}\rm
If
$$
a(x,\nabla u) := |\nabla u |^{p-2}\nabla u ,\quad p\geq 2,
$$
for every $x\in\bar\Omega$, the geometrical constants in the main result
assume simpler expressions
$$
{G}_1:=\left(\frac{2^p(2^N-1)}{\tau^p\displaystyle\min_{x\in\bar\Omega}k(x)}\right)\frac{\|k\|_{\infty}c_1}{p^{\frac{p-1}{p}}},\,\,\,\,\,\,\,\,
{G}_2:=\left(\frac{2^p(2^N-1)}{\tau^p\displaystyle\min_{x\in\bar\Omega}k(x)}\right)\frac{\|k\|_{\infty}c_q^{q}}{qp^{\frac{p-q}{p}}}.
$$
In this setting, under the same hypotheses of Theorem \ref{Main1}, for each parameter $\lambda$ belonging to
$$
\widetilde{\Lambda}_{(\gamma,\delta)}:=\frac{2^p(2^N-1)}{p\tau^p\displaystyle\min_{x\in\bar\Omega}k(x)}\left]\frac{\displaystyle \delta^p}{\displaystyle  F(\delta)},\frac{1}{\Big(a_1\displaystyle\frac{G_1}{\gamma^{p-1}}+a_2G_2\gamma^{q-p}\Big)}\right[,
$$
\noindent the following problem
\begin{equation} \tag{$\widehat{D}_\lambda^{f}$} \label{N45}
\left\{
\begin{array}{ll}
-\Delta_pu= \lambda k(x)f(u) & \rm in \quad \Omega
\\ u=0  & \rm on \,
\partial \Omega,\\
\end{array}
\right.
\end{equation}
where $\Delta_p:=\div(|\nabla u|^{p-2}\nabla u)$
denotes the $p$-Laplacian operator, possesses at least three weak solutions. The case $p=2$ has been studied in \cite[Theorem 3.1]{BonannoMolica} by using a similar variational approach. See also the work \cite{DM} for the Neumann setting.
\end{remark}

\begin{remark}\label{B}
\rm{
An explicit upper bound for the constants $c_q$ in Theorem \ref{Main1} can be easily obtained. Indeed, let $1<p<N$ and fix $q\in [1,p^{*}[$. By formula \eqref{TA} one has
$$
 c_q\leq \frac{\meas(\Omega)^{\frac{p^*-q}{p^*q}}}{N\sqrt{\pi}}\left(\frac{N!\Gamma\left(\di\frac{N}{2}\right)}{2\Gamma\left(\di\frac{N}{p}\right)\Gamma\left(\di N+1-\frac{N}{p}\right)}\right)^{1/N}\eta^{1-1/p},
$$
where
$$
\eta:=\displaystyle\frac{N(p-1)}{N-p}.
$$
On the other hand, if $p>N$, due to inequality \eqref{dT2}, we also have
$$
c_q\leq \frac{N^{-\frac{1}{p}}}{\sqrt{\pi}}\left[\Gamma\left(1+\frac{N}{2}\right)\right]^{\frac{1}{N}}\left(\frac{p-1}{p-N}\right)^{1-\frac{1}{p}}(\meas(\Omega))^{\frac{1}{N}+\frac{p-q}{qp}}.
$$
}
\end{remark}

We also have the following multiplicity result that can be proved by Theorem \ref{Main1}.

\begin{corollary}\label{main3}
Assume that $f:\erre\rightarrow\erre$ is a continuous and nonnegative function such that
$$
F(s)\leq a_2 s^{q-p},\,\,\,\,\,(\forall\;s\in \erre^+)
$$
for some $a_2>0$ and $p<q<p^*$. In addition,
suppose that condition $(\rm{h}_1)$ holds. Then for every
$$\displaystyle\lambda>\frac{2^p(2^N-1)\Lambda_2}{\tau^p\displaystyle\min_{x\in\bar\Omega}k(x)}\inf_{\delta>0}\frac{\displaystyle \delta^p}{\displaystyle  F(\delta)},$$
problem \eqref{N} has at least
two distinct, nontrivial weak solutions.
\end{corollary}

\begin{proof}
\noindent Since $f$ is a nonnegative function, hypothesis $(\rm{h}_0)$ clearly holds. Now, fixing $\bar\lambda$ as in the conclusion, there exists $\bar\delta>0$ such that
$$
\bar\lambda>\left(\frac{2^p(2^N-1)\Lambda_2}{\tau^p\displaystyle\min_{x\in\bar\Omega}k(x)}\right)\frac{\displaystyle \bar\delta^p}{\displaystyle  F(\bar\delta)}.
$$
Moreover, picking a positive constant $\bar\gamma$, with
$$
\bar\gamma<\min\Big\{
\frac{\bar{\delta}}{\kappa},\Big(\frac{2^p(2^N-1)\Lambda_2}{\lambda a_2\tau^pG_2\displaystyle\min_{x\in\bar\Omega}k(x)}\Big)^{\frac{1}{q-p}}\Big\},
$$
condition $(\rm{h}_2)$ is verified and, since
$$
\bar\lambda\in\frac{2^p(2^N-1)\Lambda_2}{\tau^p\displaystyle\min_{x\in\bar\Omega}k(x)}\left]\frac{\displaystyle \bar\delta^p}{\displaystyle  F(\bar\delta)},\frac{1}{a_2G_2\bar\gamma^{q-p}}\right[,
$$
\noindent Theorem \ref{Main1} ensures that problem $(D_{\bar\lambda}^{f})$ admits at least two nontrivial weak solutions.
\end{proof}
The next is a special case of the above result. 
\begin{corollary}\label{main8}
Let $2\leq p<N$ and assume that $f:\erre\rightarrow\erre$ is a continuous and nonnegative function that is $(p-1)$-sublinear at infinity, i.e.
$$
\displaystyle\lim_{|t|\rightarrow \infty}\frac{f(t)}{|t|^{p-1}}=0.
$$
Suppose that
$$
F(s)\leq a_2 s^{q-p},\,\,\,\,\,(\forall\;s\in \erre^+)
$$
for some $a_2>0$ and $p<q<pN/(N-p)$.\par
\noindent Then for every
$$\displaystyle\lambda>\frac{2^p(2^N-1)}{p\tau^p}\inf_{\delta>0}\frac{\displaystyle \delta^p}{\displaystyle \int_0^\delta f(t)dt},$$
the following Dirichlet problem
\begin{equation} \tag{$\widehat{D}_\lambda^{f}$} \label{N45}
\left\{
\begin{array}{ll}
-\Delta_pu= \lambda f(u) & \rm in \quad \Omega
\\ u=0  & \rm on \,
\partial \Omega,\\
\end{array}
\right.
\end{equation}
\noindent has at least
two distinct, nontrivial weak solutions in $W^{1,p}_0(\Omega)$.
\end{corollary}

\begin{example}\label{inintroduction}\rm
Let $\Omega\subset \erre^5$ be a domain with smooth boundary. Taking $p=4$, an example of map that satisfies all the assumptions of Corollary \ref{main8} is the nonnegative function $g:\erre\rightarrow \erre$ given by $g(t):=\log(1+t^4)$. Indeed,
$$
\displaystyle\lim_{|t|\rightarrow \infty}\frac{\log(1+t^4)}{|t|^{3}}=0,
$$
\noindent and its potential
\begin{eqnarray*}
G(s)
& = & \sqrt{2}\arctan (\sqrt{2}s-1)+  \sqrt{2}\arctan (\sqrt{2}s+1)\nonumber\\
& - & \frac{\sqrt{2}}{2}\log\left(\frac{s^2-\sqrt{2}s+1}{s^2+\sqrt{2}s+1}\right)+ s\log(1+s^4)-4s,\nonumber
\end{eqnarray*}
\noindent satisfies $G(s)\leq s^2$, for every $s\in\erre$.\par
\noindent Then, for every
$$\displaystyle\lambda>\frac{124}{\tau^4}\inf_{\delta>0}\frac{\displaystyle \delta^4}{\displaystyle \int_0^\delta g(t)dt}\approx \frac{752}{\tau^4},$$
the following Dirichlet problem
\begin{equation} \tag{$\widehat{D}_\lambda^{g}$} \label{N452}
\left\{
\begin{array}{ll}
-\Delta_4u= \lambda g(u) & \rm in \quad \Omega
\\ u=0  & \rm on \,
\partial \Omega,\\
\end{array}
\right.
\end{equation}
\noindent has at least
two distinct, nontrivial weak solutions in $W^{1,4}_0(\Omega)$.
\end{example}

\indent In conclusion we present an application that is a direct consequence of Theorem \ref{Main1}.

\begin{example}\label{esempio}
\rm{ \noindent Let $(a^{ij})\in C^{0}(\bar{\Omega};\erre^{N\times N})$ be a positive definite matrix such that:
\begin{itemize}
\item[$({\beta}_1)$] $a^{ij}(x)=a^{ji}(x)$, \textit{for every} $x\in \bar \Omega$;

\item[$({\beta}_2)$] \textit{There are positive constants} $\Lambda_1$ \textit{and} $\Lambda_2$ \textit{for which}
$$
\Lambda_1|\xi|^2\leq \frac{1}{2}\sum_{i,j=1}^{N}a^{ij}(x)\xi_i\xi_j\leq \Lambda_2|\xi|^2,\quad\quad \forall x\in \bar{\Omega},\,\, \xi\in \R^N.
$$
\end{itemize}
\noindent Put
$$
A(x,\xi):=\frac{1}{2}\sum_{i,j=1}^{N}a^{i,j}(x)\xi_i\xi_j,\,\,\,\,\,\forall\; (x,\xi)\in \bar{\Omega}\times\erre^N,
$$
\noindent and consider the elliptic operator in divergent form
$$
L[u]=\sum_{i,j=1}^N\frac{\partial}{\partial x_j}\Big(a^{ij}(x)\frac{\partial u}{\partial x_i}\Big).
$$
As observed by De N\'apoli and Mariani in \cite{n1}, the map $A$ satisfies assumptions $(\alpha_1)$-$(\alpha_5)$. At this point, let $\Omega$ be a nonempty bounded open subset of the Euclidean space $(\erre^N,|\cdot|)$ with boundary of class $C^1$, and let $k:\bar\Omega\rightarrow \R$ be a positive continuous function. Further, let us fix $q\in ]2,2^*[$ and consider the function $h:\erre\rightarrow\erre$ defined by
\[
h(t):= \left\{
\begin{array}{ll}
\displaystyle 1+|t|^{q-1} & \mbox{ if\, $|t|\leq r$} \\\\
\displaystyle \frac{(1+r^{2})(1+r^{q-1})}{1+t^2} & \mbox{ if $|t|> r$},
\end{array}
\right.
\]
\noindent where $r$ is a fixed constant such that
\begin{eqnarray}\label{ineq}
r>\max\left\{\kappa,q^{\frac{1}{q-2}}(G_1+G_2)^{\frac{1}{q-2}}\right\}.
\end{eqnarray}
\noindent Clearly $h(0)\neq 0$ and $h(t)\leq (1+|t|^{q-1})$ for every $t\in\erre$. Hence, condition \eqref{S} is satisfied for $a_1=a_2=1$. Moreover, also condition $(\textrm{h}_1)$ is verified since $\displaystyle\lim_{|t|\rightarrow +\infty}h(t)/|t|=0$. Finally, owing to (\ref{ineq}), it follows that
$$
G_1+G_2<\frac{r^{q-2}}{q}.
$$
Therefore,
$$
\frac{H(r)}{r^2}=\frac{r^{q-2}}{q}+\frac{1}{r}>G_1+G_2,
$$
where $H(r):=\displaystyle\int_0^rh(t)dt$,
and condition $(\textrm{h}_2)$ holds choosing $\delta=r$.\par
\noindent Consequently, owing to Theorem \ref{Main1}, for each parameter $\lambda$ belonging to the interval
$$
\Lambda_{\delta}:=\frac{4(2^N-1)\Lambda_2}{\tau^2\displaystyle\min_{x\in\bar\Omega}k(x)}\left]\frac{\displaystyle r^2}{\displaystyle  H(r)},\frac{1}{G_1+G_2}\right[,
$$
the following elliptic Dirichlet problem
\begin{equation} \tag{${\widetilde{D}}_\lambda^{h}$} \label{N3}
\left\{
\begin{array}{ll}
-L[u]= \lambda k(x)h(u) & \rm in \quad \Omega
\\ u=0  & \rm on \,
\partial \Omega,\\
\end{array}
\right.
\end{equation}
\noindent possesses at least three nontrivial weak solutions in $H^1_0(\Omega)$.}
\end{example}

\begin{remark}\label{onAR1}\rm{
We observe that, to the contrary to the result of Krist\'aly,
Lisei and Varga studied \cite[Theorem 2.1]{k1}, in Theorem \ref{Main1}, as for instance Example \ref{esempio} shows, we don't require the following behaviour at zero
$$
\lim_{t\rightarrow 0}\frac{f(t)}{|t|^{p-1}}=0,
$$
that automatically implies $f(0)=0$.\par
\noindent Moreover, for completeness, we emphasize that our results can be investigated also for different classes of elliptic operators in divergence form looking at the existence of at least three nontrivial weak solutions. See, for instance, the recent and interesting paper of Colasuonno, Pucci, and Varga \cite{CPV} for related topics.
}
\end{remark}

\medskip
 \indent {\bf Acknowledgements.}
 The authors warmly thanks the anonymous referees for their useful and nice comments on the manuscript. The research was supported in part by the SRA grants P1-0292-0101 and J1-4144-0101. Moreover, the first author was supported by the GNAMPA Project 2012 titled: {\it Esistenza e molteplicit\`{a} di soluzioni per problemi differenziali non lineari.}

\end{document}